\documentclass[10pt,draft]{amsart}

\usepackage{amsmath,amsfonts,amssymb}
\usepackage{amsthm}

\usepackage[cp1251]{inputenc}
\usepackage[T2A]{fontenc}
\usepackage[english]{babel}

\date{}

\theoremstyle{plain}
\newtheorem{theorem}{Theorem}[section]
\newtheorem{lemma}[theorem]{Lemma}
\newtheorem{proposition}[theorem]{Proposition}
\newtheorem{corollary}[theorem]{Corollary}

\large
\numberwithin{equation}{section}

\title[The Differential Operator $(-1)^{m}d^{2m}/dx^{2m}+V$ with V -- Distribution]
{Estimates for Periodic Eigenvalues of the Differential Operator
$\mathbf{(-1)^{m}d^{2m}/dx^{2m}+V}$ with V -- Distribution}
\author[Volodymyr Molyboga]{Volodymyr Molyboga}

\address{Department of Nonlinear Analysis \\
    Institute of Mathematics NAS Ukraine \\
    Tereshchenkivska str., 3 \\
    Kyiv\\
    Ukraine\\
    01601}
\email{molyboga@imath.kiev.ua}
\subjclass[2000]{}

\begin{document}

\begin{abstract}
The periodic eigenvalue problem for the differential operator
$(-1)^{m}d^{2m}/dx^{2m}+V$ is studied for complex-valued
distribution V in the Sobolev space
$H^{-m\alpha}_{per}[-1,1]\;(m\in\mathbb{N},\;
0\leq\alpha<1)$. In paticular, the case $m=1$ was investigated in
paper \cite{R2} by the same method. The following result is shown:

The periodic spectrum consists of a sequence
$(\lambda_{k})_{k\geq0}$ of complex eigenvalues satisfying the
asymptotics (for any $\varepsilon>0$)
$$
  \lambda_{2n-1},\lambda_{2n}=n^{2m}\pi^{2m}+\hat{V}(0)\pm
  \sqrt{\hat{V}(-2n)\hat{V}(2n)}+o(n^{m(2\alpha-1+\varepsilon)}),
$$
where $\hat{V}(k)$ denote the Fourier coefficients of V.
\end{abstract}
\maketitle
\section{introduction}\label{int}
Let $m\in\mathbb{N}$, we consider the eigenvalue problem on the interval
$[-1,1]$,
$$
  (-1)^{m}\frac{d^{2m}}{dx^{2m}}y+Vy=\lambda y,
$$
where $\lambda\in\mathbb{C}$ and V is a complex-valued
distribution in the Sobolev space
$H^{-m\alpha}_{per}\equiv H^{-m\alpha}_{per}[-1,1]$ with
$0\leq\alpha\leq1$,
$$
  H^{-m\alpha}_{per}:=\left\{f=\sum_{k\in\mathbb{Z}}\hat{f}(k)e^{ik\pi
  x}\left|\;\parallel f\parallel_{H^{-m\alpha}_{per}}<\infty\right.\right\},
$$
where, with $<k>:=1+\mid k\mid$,
$$
  \parallel
  f\parallel_{H^{-m\alpha}_{per}}:=\left(\sum_{k\in\mathbb{Z}}
  <k>^{-2m\alpha}\mid\hat{f}(k)\mid ^{2}\right)^{1/2}.
$$
Similarly, as for $m=1$ and potentials V  in $L_{\mathbb{C}}^{2}[-1,1]$ \cite{R4}, it
turn out that the spectrum $\emph{spec}(L_{m})$ of the differential operator
$L_{m}:=(-1)^{m}d^{2m}/dx^{2m}+V$  with periodic boundary conditions is discrete
and consists of a sequence of eigenvalues $\lambda_{k}=\lambda_{k}(V)\;(k\geq
0)$ with the property that $Re\,\lambda_{k}\rightarrow +\infty$
for $k\rightarrow\infty$. Hence, the eigenvalues $\lambda_{k}$
are enumerated with their algebraic multiplicities and ordered so
that
$$
  Re\,\lambda_{k}<Re\,\lambda_{k+1}\qquad or \qquad
  Re\,\lambda_{k}=Re\,\lambda_{k+1}\qquad and \qquad Im\,\lambda_{k}\leq
  Im\,\lambda_{k+1}.
$$
Introduce, for $n\geq 1$
$$
  \tau_{mn}:=\frac{\lambda_{2n}+\lambda_{2n-1}}{2}; \qquad\;
  \gamma_{mn}:=\lambda_{2n}-\lambda_{2n-1}
$$
and   denote by $h^{s}\equiv h^{s}(\mathbb{N;C})$ the weighted
$l^{2}$-sequence spaces $(s\in\mathbb{R})$
$$
  h^{s}:=\left\{x=(x_{n})_{n\geq 1}\mid\;\parallel
  x\parallel_{s}<\infty\right\},
$$
where
$$
  \parallel x\parallel_{s}:=\left(\sum_{n\geq 1}<n>^{2s}\mid
  x_{n}\mid^{2}\right)^{1/2}.
$$
\begin{theorem}\label{th_1}
Let $m\in\mathbb{N},\;0\leq\alpha<1$. Then, for any $\varepsilon>0$,
uniformly for bounded sets of distributions V  in
$H^{-m\alpha}_{per}[-1,1]$, the following asymptotic estimates hold:
\begin{align*}
\mathrm{(i)}\qquad\mbox{\qquad}  & (\tau_{mn}-n^{2m}\pi^{2m}-\hat{V}(0))_{n\geq 1}\in
  h^{m(1-2\alpha-\varepsilon)} \\
\mathrm{(ii)}\qquad\mbox{\qquad} &
\left(\min_{\pm}\left|\gamma_{mn}\pm2\sqrt{\hat{V}(-2n)\hat{V}(2n)}\right|\right)_{n\geq 1}\in
  \begin{cases}
    h^{m(1/2-\alpha)} & \text{if}\quad 0\leq\alpha<\frac{1}{2}; \\
    h^{m(1-2\alpha-\varepsilon)} & \text{if}\quad \frac{1}{2}\leq\alpha<1.
  \end{cases}
\end{align*}
\end{theorem}

As a consequence of Theorem \ref{th_1}, one obtains
\begin{corollary}\label{cor_1}
Let $m\in\mathbb{N}$, $0\leq\beta<\alpha<1$. Assume that the
distribution V in $H^{-m\alpha}_{per}[-1,1]$ is of period 1 and real
valued (i.e. $\hat{V}(2k+1)=0$, $\hat{V}(-k)=\overline{\hat{V}(k)}$ $\forall
k\in\mathbb{Z}$). Then $V\in H^{-m\beta}_{per}[-1,1]$ iff $(\gamma_{mn})_{n\geq 1}\in
h^{-m\beta}$.
\end{corollary}

The notation used is standard. Given a Banach space E, we denote
by $\mathcal{L}(E)$ the Banach space of linear bounded operators on E.
For $s\in\mathbb{R}$, the space $H^{s}_{per}\equiv H^{s}_{per}[-1,1]$ denotes the
Sobolev space
$$
  H^{s}_{per}:=\left\{f=\sum_{k\in\mathbb{Z}}\hat{f}(k)e^{ik\pi
  x}\left|\;\parallel f\parallel_{H^{s}_{per}}<\infty\right.\right\},
$$
where
$$
  \parallel
  f\parallel_{H^{s}_{per}}:=\left(\sum_{k\in\mathbb{Z}}
  <k>^{2s}\mid\hat{f}(k)\mid ^{2}\right)^{1/2}.
$$
The Fourier coefficient $\hat{f}(k)$ are defined by the formula
$$
  \hat{f}(k):=<f,e^{ik\pi x}>,
$$
where $<\cdot,\cdot>$ denotes the sesquilinear pairing between $H^{s}_{per}$
and $H^{-s}_{per}$ extending the $L^{2}$-inner product
$$
  <g,h>:=\frac{1}{2}\int_{-1}^{1}g(x)\overline{h(x)}\,dx, \quad
  g,h\in L^{2}[-1,1].
$$
The following weighted $l^{2}$-spaces will be used: For any
$K\subseteq\mathbb{Z}$, $n\in\mathbb{Z}$, and $s\in\mathbb{R}$, denote
by $h^{s,n}\equiv h^{s,n}(K;\mathbb{C})$ (will be consider in Appendix below)
the Hilbert space of
sequences $(a(j))_{j\in K}\subseteq\mathbb{C}$ with inner product
$<a,b>_{s,n}:=\sum_{j\in K}<j+n>^{2s}a(j)\overline{b(j)}$. The
norm of $a:=(a(j))_{j\in K}$ in $h^{s,n}$ is denoted by $\parallel
a\parallel_{h^{s,n}}$. For n=0, we simply write $h^{s}$ instead of
$h^{s,0}$. To shorten notation, it is convenient to denote by
$h^{s}(n)$ the n-th element of a sequence $a:=(a(j))_{j\in K}$ in
$h^{s}$. Further we denote by $h^{s}_{0}\equiv
h^{s}_{0}(\mathbb{Z};\mathbb{C})$ the subspace of
$h^{s}(\mathbb{Z};\mathbb{C})$ defined by
$$
  h^{s}_{0}:=\left\{(a(j))_{j\in
  \mathbb{Z}}\left|\;a(0)=0\right.\right\}.
$$
Notice that adding a constant to a distribution V shifts the spectrum of
$(-1)^{m}d^{2m}/dx^{2m}+V$ by the same constant. Throughout a
remainder of the paper we therefore assume, without loss of
generality, that
\begin{equation}\label{eq_1}
  \hat{V}(0)=0.
\end{equation}

\section{periodic spectrum of $(-1)^{m}d^{2m}/dx^{2m}+V$}\label{sec_1}
For $m\in \mathbb{N}$, $V\in H^{-m\alpha}_{per}$ $(0\leq\alpha\leq
1)$, denote by $L_{m}$ the differential operator
$(-1)^{m}d^{2m}/dx^{2m}+V$ on $H^{-m\alpha}_{per}$ with domain
$H^{m(2-\alpha)}_{per}$. It is convinient to deal with the Fourier
transform $\hat{L}_{m}$ of $L_{m}$. It is of the form
$\hat{L}_{m}=D_{m}+B$, where $D_{m}$ and $B$ are infinite
matrices,
$$
  D_{m}(k,j):=k^{2m}\pi^{2m}\delta_{kj},\qquad B(k,j):=v(k-j),
$$
where $v(k):=\hat{V}(k)$ being a sequence in $h^{-m\alpha}_{0}\equiv
h^{-m\alpha}_{0}(\mathbb{Z};\mathbb{C})$ by assumption
\eqref{eq_1}. By the convolution lemma (see Appendix), $Ba=v*a$ is
well defined for $a\in h^{m(2-\alpha)}$ and hence $D_{m}+B$ is an
operator on $h^{-m\alpha}$ with domain $h^{m(2-\alpha)}$. We want
to compare the spectrum \emph{spec}($D_{m}+B$) of the operator
$D_{m}+B$ with the spectrum of $D_{m}$, \emph{spec}
($D_{m})=\{k^{2m}\pi^{2m}\;\left|\;k\in
\mathbb{Z}_{\geq 0}\right.\}$. For this purpose write for
$\lambda\in\mathbb{C}\backslash\emph{spec}(D_{m})$,
$$
  \lambda-D_{m}-B=D^{1/2}_{m\lambda}(I_{m\lambda}-S_{m\lambda})D^{1/2}_{m\lambda},
$$
where $D^{1/2}_{m\lambda},\;I_{m\lambda}$ and $S_{m\lambda}$ are
the following infinite matrices $(k,j\in\mathbb{Z})$
$$
  D_{m\lambda}(k,j):=|\lambda-k^{2m}\pi^{2m}|\delta_{kj};\qquad
  I_{m\lambda}(k,j):=\frac{\lambda-k^{2m}\pi^{2m}}{|\lambda-k^{2m}\pi^{2m}|}\delta_{kj};
$$
$$
  S_{m\lambda}(k,j):=\frac{v(k-j)}{|\lambda-k^{2m}\pi^{2m}|^{1/2}|\lambda-j^{2m}\pi^{2m}|^{1/2}}.
$$
Notice that $D^{1/2}_{m\lambda}$ and $I_{m\lambda}$ are diagonal
matrices independent on $v$. Both $I_{m\lambda}$ and $S_{m\lambda}$
can be viewed as linear operators on $h^{0}$. It is clearly that
if $\lambda\in\mathbb{C}\backslash\emph{spec}(D_{m})$ with
$\|S_{m\lambda}\|_{\mathcal{L}(h^{0})}<1$, then it is in the resolvent set \emph{Resol}($v$)
of $D_{m}+B$, and for such a $\lambda$,
\begin{equation}\label{eq_2}
  (\lambda-D_{m}-B)^{-1}=D^{-1/2}_{m\lambda}(I_{m\lambda}-S_{m\lambda})^{-1}D^{-1/2}_{m\lambda},
\end{equation}
where the right side of \eqref{eq_2} is viewed as a composition
$$
  h^{-m\alpha}\rightarrow h^{0}\rightarrow
  h^{0}\rightarrow h^{m}\; (\hookrightarrow h^{-m\alpha}).
$$

For a suitable choice of parameters, the following regions of
$\mathbb{C}$ will turn out to be in \emph{Resol}($v$): Given $M>0$,
denote by $Ext_{M}$ the exterior domain of a cone,
$$
  Ext_{M}:=\{\lambda\in\mathbb{C}\;|\;Re\,\lambda\leq|Im\,\lambda|-M\},
$$
and, for $n\geq1$ and $0<r<n^{m}\pi^{2m}\;(m\in\mathbb{N})$, by
$Vert^{m}_{n}(r)$ a vertical strip with a disk around
$n^{2m}\pi^{2m}$ of radius r removed,
$$
  Vert^{m}_{n}(r):=\{\lambda=n^{2m}\pi^{2m}+z\in\mathbb{C}\;|\;|Re\,z| \leq
  n^{m}\pi^{2m};\;|z|\geq r\}.
$$

In a straightforward way one can prove
\begin{lemma}\label{l_1}
  Let $m\in\mathbb{N}$, $0\leq\alpha\leq1,\;M\geq1$, and $v\in
  h^{-m\alpha}_{0}$. Then, for any $\lambda\in Ext_{M}$,
  $$
    \|S_{m\lambda}\|_{\mathcal{L}(h^{0})}\leq 2^{2m+1}\,\|v\|_{h^{-m\alpha}}
    \frac{1}{M^{(1-\alpha)/2+1/4}}.
  $$
\end{lemma}
\begin{proof}
  We estimate the $\mathcal{L}(h^{0})$-norm of $S_{m\lambda}$ by its
  Hilbert-Schmidt norm. Using
  $<k-j>^{2m\alpha}\leq4^{m}(<k>^{2m\alpha}+<j>^{2m\alpha})\;(k,j\in\mathbb{Z})$
  together with the trigonometric estimate
  $|\lambda-k^{2m}\pi^{2m}|\geq(M+k^{2m}\pi^{2m})\sin(\frac{\pi}{4})\; (k\in\mathbb{Z},
  \lambda\in Ext_{M})$, one concludes
\begin{align*}
    \|S_{m\lambda}\|_{\mathcal{L}(h^{0})}
   & \leq\left(\sum_{k,j}\frac{|v(k-j)|^{2}}{|\lambda-
     k^{2m}\pi^{2m}||\lambda-j^{2m}\pi^{2m}|}\right)^{1/2} \\
   & \leq\left(2\sum_{k,j} \frac{4^{m}(<k>^{2m\alpha}+<j>^{2m\alpha})}{(M+k^{2m}\pi^{2m})
     (M+j^{2m}\pi^{2m})}<k-j>^{-2m\alpha}|v(k-j)|^{2}\right)^{1/2} \\
   & \leq\left(4^{m+1}\sup_{k}\frac{<k>^{2m\alpha}}{(M+k^{2m}\pi^{2m})}\sum_{k}\frac{1}
     {(M+k^{2m}\pi^{2m})}\sum_{j}<k-j>^{-2m\alpha}|v(k-j)|^{2}\right)^{1/2} \\
   & =2^{m+1}\left(\sup_{k}\frac{<k>^{2m\alpha}}{(M+k^{2m}\pi^{2m})}\right)^{1/2}
     \left(\sum_{k}\frac{1}{(M+k^{2m}\pi^{2m})}\right)^{1/2}\|v\|_{h^{-m\alpha}}.
\end{align*}
\end{proof}

For $\lambda\in Vert^{m}_{n}(r)$, the following estimate for
$\|S_{m\lambda}\|_{\mathcal{L}(h^{0})}$ can be obtained:
\begin{lemma}\label{l_2}
  Let $m\in\mathbb{N}$,
  $0\leq\alpha\leq1$, $n\geq\frac{(8m-4)m}{8m-7}$, $0<r<n^{m}\pi^{2m}$, and $v\in
  h^{-m\alpha}_{0}$. Then, for any $\lambda\in Vert^{m}_{n}(r)$,
  $$
    \|S_{m\lambda}\|_{\mathcal{L}(h^{0})}\leq\frac{1}{r}\left(|v(2n)|+|v(-2n)|\right)+
    4\left[\frac{2}{\pi}\right]^{m}\left(\frac{n^{m(\alpha-1+1/2m)}}{\sqrt{r}}+
    \frac{6\log n}{n^{m(1-\alpha)}}\right)\|v\|_{h^{-m\alpha}}.
  $$
\end{lemma}

To prove Lemma \ref{l_2}, one uses that $(\lambda\in Vert^{m}_{n}(r)$,
$n\geq\frac{(8m-4)m}{8m-7}$, $k\neq\pm n)$
\begin{equation}\label{eq_3}
  \frac{1}{|\lambda-k^{2m}\pi^{2m}|}\leq\frac{3}{\pi^{2m}}\frac{1}{|k^{2m}-n^{2m}|}
\end{equation}
together with the following elementary estimates
\begin{lemma}\label{l_3}
  Let $m\in\mathbb{N}$, $0\leq\alpha\leq1$, and $n\geq m$. Then \\
$\qquad\mathrm{(a)}\quad \sup_{k\neq\pm n}\frac{<k>^{m\alpha}}{|k^{2m}-n^{2m}|^{1/2}}\leq
3^{m\alpha}n^{m(\alpha-1+\frac{1}{2m})};$ \\
$\qquad\mathrm{(b)}\quad \sup_{k\neq\pm n}\frac{<k\pm n>^{m\alpha}}{|k^{2m}-n^{2m}|^{1/2}}\leq
4^{m\alpha}n^{m(\alpha-1+\frac{1}{2m})};$ \\
$\qquad\mathrm{(c)}\quad \sum_{k\neq\pm n}\frac{1}{|k^{2m}-n^{2m}|^{1/2}}\leq
    5\frac{1+\log n}{n}.$
\end{lemma}

Lemma \ref{l_2} together with
$$
  \left(|v(2n)|+|v(-2n)|\right)\leq 3^{m}\sqrt{2}\,\|v\|_{h^{-m\alpha}}n^{m\alpha}
$$
leads to
$$
  \|S_{m\lambda}\|_{\mathcal{L}(h^{0})}\leq\frac{3^{m}\sqrt{2}n^{m\alpha}\|v\|_{h^{-m\alpha}}}
  {r}+4\left[\frac{2}{\pi}\right]^{m}\left(\frac{n^{m(\alpha-1+1/2m)}}{\sqrt{r}}+
  \frac{6\log n}{n^{m(1-\alpha)}}\right)\|v\|_{h^{-m\alpha}}.
$$

Combining this with Lemma \ref{l_1} one obtains
\begin{proposition}\label{pr_1}
  Let $m\in\mathbb{N}$, $0\leq\alpha<1,\; \mathrm{R}>0,\; \mathrm{C}>2$, and
  $r_{n}:=3^{m}\sqrt{2}\mathrm{C}\mathrm{R}n^{m\alpha}\;(n\geq1)$.
  Then there exist $\mathrm{M}\geq1$ and $n_{0}\geq\frac{(8m-4)m}{8m-7}$ with
  $0<r_{n}<n^{m}\pi^{2m}\;\forall n\geq n_{0}$ so that, for any $v\in
  h^{-m\alpha}_{0}$ with $\|v\|_{h^{-m\alpha}}\leq \mathrm{R}$
  $$
    \|S_{m\lambda}\|_{\mathcal{L}(h^{0})}<1\quad\mbox{for}\quad\lambda\in\mathrm{Ext_{M}}
    \cup\bigcup_{n\geq n_{0}}Vert^{m}_{n}(r_{n}).
  $$
    Hence
  $$
    \mathrm{Ext_{M}}\cup\bigcup_{n\geq
    n_{0}}Vert^{m}_{n}(r_{n})\subseteq\emph{Resol}\,(v).
  $$
\end{proposition}

In view of the formula \eqref{eq_1} and the fact that the
imclusion $h^{m}\hookrightarrow h^{-m\alpha}$ is compact, the
operator $D_{m}+B$ has a compact resolvent. Hence
\emph{spec}$(D_{m}+B)$ consists of an (at most) countable set of
eigenvalues of finite multiplicity, contained in the complement of
$\mathrm{Ext_{M}}\cup\bigcup_{n\geq n_{0}}Vert^{m}_{n}(r_{n})$,
where $\mathrm{C,\;M,\;n_{0},\;r_{n}}$ and $v$ are given as in
Proposition \ref{pr_1}. To localize this eigenvalues notice that
for any $0\leq s\leq1$
$$
  \mathrm{Ext_{M}}\cup\bigcup_{n\geq
  n_{0}}Vert^{m}_{n}(r_{n})\subseteq\emph{Resol}\,(sv).
$$
Hence, for any contour $\Gamma\subseteq\mathrm{Ext_{M}}\cup\bigcup_{n\geq
n_{0}}Vert^{m}_{n}(r_{n})$ and any $0\leq s\leq1$, the Riesz projector
$$
  P(s):=\frac{1}{2\pi
  i}\int_{\Gamma}(\lambda-D_{m}-sB)^{-1}\,d\lambda\in\mathcal{L}(h^{-m\alpha})
$$
is well defined and the dimension of its range is independent of
$s$. Therefore the number of eigenvalues of $D_{m}+B$ and $D_{m}$
inside $\Gamma$ (counted with their algebraic multiplicities) are
the same.

Further, if the sequence $v\in h^{-m\alpha}_{0}$ satisfies the
condition $v(-k)=\overline{v(k)}\;\forall k\in\mathbb{Z}$, the
eigenvalues of the operator $D_{m}+B$ are real. This is shown in a
standard fashion by taking the $h^{0}$-inner product of the
eigenvalue equation with the corresponding eigenvector.
Summarizing the result above, we obtain the following
\begin{proposition}\label{pr_2}
  Let $m\in\mathbb{N}$, $0\leq\alpha<1,\; \mathrm{C}>2$, and
  $\mathrm{R}>0$. Then there exist $\mathrm{M}\geq1$ and
  $n_{0}\in\mathbb{N}$ so that, for any $v\in h^{-m\alpha}_{0}$ with $\|v\|_{h^{-m\alpha}}\leq
  \mathrm{R}$, \emph{spec}$(D_{m}+B)$ consists of precisely
  $2n_{0}-1$ eigenvalues inside the bounded cone
  $$
    \left\{\lambda\in\mathbb{C}\,\left|\,\right.|Im\,\lambda|-M\leq Re\,\lambda \leq
    (n_{0}^{2m}-n_{0}^{m})\pi^{2m}\right\},
  $$
  and a sequence of pairs of eigenvalues $\lambda_{n}^{+},\;\lambda_{n}^{-}\;(n\geq
  n_{0})$ inside a disc around $n^{2m}\pi^{2m}$,
  $$
    |\lambda_{n}^{\pm}-n^{2m}\pi^{2m}|<3^{m}\sqrt{2}\mathrm{C}\mathrm{R}n^{m\alpha}.
  $$
  If, an addition, $v$ satisfies $v(-k)=\overline{v(k)}\;\forall
  k\in\mathbb{Z}$, then all eigenvalues are real.
\end{proposition}

List the eigenvalues $\lambda_{0},\;\lambda_{1},\;\lambda_{2},\;\ldots$
of $(D_{m}+B)$ (counted with their algebraic multiplicities) in such
a way that
$$
  Re\,\lambda_{j}<Re\,\lambda_{j+1}\qquad or \qquad
  Re\,\lambda_{j}=Re\,\lambda_{j+1}\qquad and \qquad Im\,\lambda_{j}\leq
  Im\,\lambda_{j+1}.
$$
Then, by Proposition \ref{pr_2}, for n large $\{\lambda_{2n-1},\;\lambda_{2n}\}=
\{\lambda_{n}^{+},\;\lambda_{n}^{-}\}$ satisfy
$$
  \lambda_{n}^{\pm}=n^{2m}\pi^{2m}+O(n^{m\alpha}).
$$

One can prove that the following statement is volid.
\begin{lemma}\label{l_4}
  Let $m\in\mathbb{N}$, $0\leq\alpha<1,\; \mathrm{R}>0,\; \mathrm{C}>2,\; \mathrm{M}_{0}\geq1$,
  $n_0\geq\frac{(8m-4)m}{8m-7}$.
  Then for any $v\in h^{-m\alpha}_{0}$ with $\|v\|_{h^{-m\alpha}}\leq
  \mathrm{R}$ the following statementes hold:
\begin{align*}
\mathrm{(a)}\; & \quad\mathrm{M}_{0}^{(1-\alpha)/2+1/4}>2^{2m+2}\mathrm{R}\Longrightarrow
\|S_{m\lambda}\|_{\mathcal{L}(h^{0})}<\frac{1}{2}\quad\mbox{for}\quad\lambda\in\mathrm{Ext_{M}},
               \quad \forall M\geq M_0; \\
\mathrm{(b)}\; & \quad n_{0}^{m(1-\alpha)/2}>\frac{4\mathrm{C}}{\mathrm{C}-2}\left[\frac{2}{\pi}
\right]^{m}\left(\frac{2^{m-1/4}}{3^{m/2}\mathrm{C}^{1/2}\mathrm{R}^{1/2}n^{m(1-1/m)/2}}+
\frac{12}{m(1-\alpha)}\right)\Longrightarrow\|S_{m\lambda}\|_{\mathcal{L}(h^{0})}<\frac{1}{2} \\
               & \quad\mbox{for}\;\lambda\in\mathrm{Vert^{m}_{n}(r_{n})}, \quad \forall n\geq
               n_0.
\end{align*}
\end{lemma}

In the following section we improve on asymptotics
$\lambda_{n}^{\pm}=n^{2m}\pi^{2m}+O(n^{m\alpha})$. For this
purpose we will consider vertical strips
$\mathrm{Vert^{m}_{n}(r_{n})}$ with a circle of radius $r=n^{m}$
around $n^{2m}\pi^{2m}$ removed, and the following estimate for
the operators $S_{m\lambda}$ will be useful:
\begin{lemma}\label{l_5}
  Let $m\in\mathbb{N}$, $0\leq\alpha<1$, and $\varepsilon>0$. Then
  there exists $C=C(\alpha,\varepsilon)$ such that, for any $v\in h^{-m\alpha}_{0}$
  $$
    \left\|\left(\sup_{\lambda\in \mathrm{Vert^{m}_{n}(r_{n})}}\|S_{m\lambda}\|
    _{\mathcal{L}(h^{0})}\right)_{n\geq1}\right\|_{h^{m(1-\alpha-\varepsilon)}}\leq
    C\|v\|_{h^{-m\alpha}}.
  $$
\end{lemma}
\begin{proof}
For any given $n\geq1$, split the operator $S_{m\lambda}$,
$S_{m\lambda}=\sum_{j=1}^{6}I_{A_{n}^{(j)}}S_{m\lambda}$, where
$I_{A}:\mathbb{Z}\times\mathbb{Z}\rightarrow\mathbb{R}$
denotes the characteristic function of a set
$A\subseteq\mathbb{Z}\times\mathbb{Z}$ and $A^{(j)}\equiv
A^{(j)}_{n}$  is the following decomposition of
$\mathbb{Z}\times\mathbb{Z}$
\begin{align*}
 & A^{(1)}:=\{(k,j)\,\left|\right.\,k,j\in\{\pm n\}\};& A^{(2)}:=\{(k,j)\,\left|\right.
  \,k,j\in\mathbb{Z}\backslash\{\pm n\}\}; \\
 & A^{(3)}:=\{(k,j)\,\left|\right.\,k=n,j\neq\pm n\}; & A^{(4)}:=\{(k,j)\,\left|\right.
  \,k=-n,j\neq\pm n\}; \\
 & A^{(5)}:=\{(k,j)\,\left|\right.\,k\neq\pm n,j=n\}; & A^{(6)}:=\{(k,j)\,\left|\right.
  \,k\neq\pm n,j=-n\}.
\end{align*}
Then
$$
  \sup_{\lambda\in \mathrm{Vert^{m}_{n}(n^{m})}}\|S_{m\lambda}\|
  _{\mathcal{L}(h^{0})}\leq\sum_{j=1}^{6}\sup_{\lambda\in \mathrm{Vert^{m}_{n}(n^{m})}}
  \|I_{A_{n}^{(j)}}S_{m\lambda}\|_{\mathcal{L}(h^{0})}
$$
and each term in the latter sum is treated separately.

As $v(0)=0$, we have, for any $\lambda\in \mathrm{Vert^{m}_{n}(n^{m})}$
\begin{align*}
   \|I_{A_{n}^{(1)}}S_{m\lambda}\|_{\mathcal{L}(h^{0})}&\leq
    \left(\sum_{(k,j)=\pm (n,-n)}\frac{|v(k-j)|^{2}}{|\lambda-
    k^{2m}\pi^{2m}||\lambda-j^{2m}\pi^{2m}|}\right)^{1/2} \\
  & \leq\frac{1}{n^{m}}\left(|v(2n)|^{2}+|v(-2n)|^{2}\right)^{1/2}.
\end{align*}

The operators $I_{A_{n}^{(j)}}S_{m\lambda}$ for $3\leq j\leq6$ are
estimated similar. Let us consider e.g. the case $j=3$. For
non-negative sequences $a,b,c\in h^{0}$ with $a(n)=0\;\forall
n\geq0$, we have
\begin{align*}
 & \sum_{n}a(n)<n>^{m(1-\alpha)}\sup_{\lambda\in \mathrm{Vert^{m}_{n}(n^{m})}}
  \left|\sum_{k}\left(I_{A_{n}^{(3)}}S_{m\lambda}b\right)(k)
  \cdot c(k)\right|  \\
 & \leq\sum_{n\geq1}\sum_{k=n,j\neq\pm n}a(n)<n>^{m(1-\alpha)}\sup_{\lambda\in
 \mathrm{Vert^{m}_{n}(n^{m})}}\frac{|v(k-j)|}{|\lambda-
   k^{2m}\pi^{2m}|^{1/2}|\lambda-j^{2m}\pi^{2m}|^{1/2}}\,b(j)\,c(k) \\
 & \leq\frac{\sqrt{3}}{\pi^{m}}\sup_{n\geq1,j\neq\pm n}\frac{<n>^{m(1-\alpha)}<n-j>^{m\alpha}
 }{n^{m/2}|j^{2m}-n^{2m}|^{1/2}}\sum_{n\geq1,j\neq\pm
 n}\frac{|v(n-j)|}{<n-j>^{m\alpha}}\,a(n)\,b(j)\,c(n),
\end{align*}
where for the last inequality we use \eqref{eq_2}.

By the Cauchy-Schwartz inequality and the following estimate
obtained from Lemma \ref{l_3}(b)
$$
 \sup_{n\geq1,j\neq\pm n}\frac{<n>^{m(1-\alpha)}<n-j>^{m\alpha}
 }{n^{m/2}|j^{2m}-n^{2m}|^{1/2}}\leq 4^{m}n^{(-1/2+1/2m)}
$$
one gets
\begin{align*}
  \sum_{n}a(n)&<n>^{m(1-\alpha)}\sup_{\lambda\in \mathrm{Vert^{m}_{n}(n^{m})}}
                  \left|\sum_{k}\left(I_{A_{n}^{(3)}}S_{m\lambda}b\right)(k)\cdot
                   c(k)\right| \\
              & \leq\frac{4^{m}\sqrt{3}}{\pi^{m}}
                                 \,\|v\|_{h^{0}}\,\|a\|_{h^{0}}
                                 \,\|b\|_{h^{0}}\,\|c\|_{h^{0}}.
\end{align*}

It remains to estimate
$\|I_{A_{n}^{(2)}}S_{m\lambda}\|_{\mathcal{L}(h^{0})}$. For
$a,\;b,\;c$ in $h^{0}$ as above, and any $\varepsilon>0$ (w.l.o.g. we assume
$1-\alpha-\varepsilon\geq0$), one obtains in the same fashion as above
\begin{align*}
    \sum_{n}a(n)&<n>^{m(1-\alpha-\varepsilon)}\sup_{\lambda\in \mathrm{Vert^{m}_{n}(n^{m})}}
     \left|\sum_{k}\left(I_{A_{n}^{(2)}}S_{m\lambda}b\right)(k)\cdot
     c(k)\right| \\
   & \leq\frac{{3}}{\pi^{2m}}\sum_{n\geq1,k,j\neq\pm n}R_{m}(n,k,j)
   \frac{|v(k-j)|}{<k-j>^{m\alpha}}\,a(n)\,b(j)\,c(k),
\end{align*}
where
$$
  R_{m}(n,k,j):=\frac{<n>^{m(1-\alpha-\varepsilon)}<k-j>^{m\alpha}}
  {|k^{2m}-n^{2m}|^{1/2}|j^{2m}-n^{2m}|^{1/2}}.
$$
The latter sum is estimated using Cauchy-Schwartz inequality. To
estimate $R_{m}(n,k,j)$, we split up $A_{n}^{(2)}=\{(k,j)\in\mathbb{Z}^{2}\,\left|\right.
\,k,j\neq\pm n\}$. First notice that, as $R_{m}(n,k,j)$ is
symmetric in $k$ and $j$, it suffices to consider the case
$|j|\leq|k|$. Then, using $<k-j>^{m\alpha}\leq2^{m\alpha}<k>^{m\alpha}$ and
$<n>^{m(1-\alpha-\varepsilon)}\leq
2^{m(1-\alpha-\varepsilon)}<n>^{m(1-\alpha-\varepsilon)}$,
$$
  R_{m}(n,k,j)\leq 2^{m}\frac{<n>^{m(1-\alpha-\varepsilon)}<k>^{m\alpha}}
  {|k^{2m}-n^{2m}|^{1/2}|j^{2m}-n^{2m}|^{1/2}}.
$$
For the subsets $A_{n}^{(\pm,\pm)}\cap\{|j|\leq|k|\leq 2n\}$ of
$A_{n}^{(2)}$,
$$
  A_{n}^{(\pm,\pm)}:=\{(k,j)\in A_{n}^{(2)}\,\left|\right.\,\pm k\geq0;\;\pm
  j\geq0\},
$$
we argue similarly. Consider e.g. $A_{n}^{(-,+)}$. Then
$|k|+n=|k-n|$ and $|j|+n=|j+n|$, hence
$$
  <n>^{m(1-\alpha-\varepsilon)}<k>^{m\alpha}\leq\,<n>^{m(1-\alpha-\varepsilon)}
  (1+2n)^{m\alpha}\leq3^{m\alpha}n^{m(1-\varepsilon)}:
$$
1) $m=2l+1,\;l\in\mathbb{N}$
\begin{align*}
    <n>^{m(1-\alpha-\varepsilon)}&<k>^{m\alpha}\leq 3^{m\alpha}n^{m(1-\varepsilon)}\leq
   3^{m\alpha}|k^{m}-n^{m}|^{(1-\varepsilon)/2}|j^{m}+n^{m}|^{(1-\varepsilon)/2} \\
   & \leq3^{m\alpha}|k^{m}-n^{m}|^{1/2}|k^{m}+n^{m}|^{-\varepsilon/2}
     |j^{m}+n^{m}|^{1/2}|j^{m}-n^{m}|^{-\varepsilon/2},
\end{align*}
which leads to
$$
  R_{m}(n,k,j)\leq
  6^{m}|k^{m}+n^{m}|^{-(1+\varepsilon)/2}|j^{m}-n^{m}|^{-(1+\varepsilon)/2};
$$
2) $m=2l,\;l\in\mathbb{N}$
\begin{align*}
    <n>^{m(1-\alpha-\varepsilon)}&<k>^{m\alpha}\leq 3^{m\alpha}n^{m(1-\varepsilon)}\leq
   3^{m\alpha}|k^{m}+n^{m}|^{(1-\varepsilon)/2}|j^{m}+n^{m}|^{(1-\varepsilon)/2} \\
   & \leq3^{m\alpha}|k^{m}+n^{m}|^{1/2}|k^{m}-n^{m}|^{-\varepsilon/2}
     |j^{m}+n^{m}|^{1/2}|j^{m}-n^{m}|^{-\varepsilon/2},
\end{align*}
which leads to
$$
  R_{m}(n,k,j)\leq
  6^{m}|k^{m}-n^{m}|^{-(1+\varepsilon)/2}|j^{m}-n^{m}|^{-(1+\varepsilon)/2}.
$$
Therefore
$$
  R_{m}(n,k,j)\leq
  6^{m}|(-1)^{m+1}k^{m}+n^{m}|^{-(1+\varepsilon)/2}|j^{m}-n^{m}|^{-(1+\varepsilon)/2}.
$$
By the Cauchy-Schwartz inequality one then gets
\begin{align}\label{eq_4}
  & \sum_{n\geq1}\sum_{k,j\neq\pm n}|(-1)^{m+1}k^{m}-n^{m}|^{-(1+\varepsilon)/2}
    |j^{m}-n^{m}|^{-(1+\varepsilon)/2}\frac{|v(k-j)|}{<k-j>^{m\alpha}}\,a(n)\,b(j)\,c(k)
    \\ \notag
  & \leq\left(\sum_{n\geq1}a^{2}(n)\sum_{j\neq\pm n}|j^{m}-n^{m}|^{-(1+\varepsilon)/2}
    \sum_{k\neq\pm n}<k-j>^{-2m\alpha}|v(k-j)|^{2}\right)^{1/2} \\ \notag
  & \leq C\,\|v\|_{h^{0}}\,\|a\|_{h^{0}}\,\|b\|_{h^{0}}\,\|c\|_{h^{0}}.
\end{align}
Next consider the subsets
$A_{n}^{(\pm,\pm)}\cap\{|j|\leq|k|;\;|k|>2n\}$ of $A_{n}^{(2)}$. Again we
argue similarly for each of these subsets. Consider e.g.
$A_{n}^{(+,+)}$. In the case $0\leq\alpha<1/2$, choose w.l.o.g.
$\varepsilon>0$ with $\frac{1}{2}-\alpha-\frac{\varepsilon}{2}\geq0$. Then
\begin{align*}
    <n>^{m(1-\alpha-\varepsilon)}&<k>^{m\alpha}\leq 2^{m\alpha}n^{m(1-\varepsilon)/2}
     n^{m(1-2\alpha-\varepsilon)/2}|k|^{m\alpha} \\
   & \leq2^{m\alpha}|j^{m}+n^{m}|^{(1-\varepsilon)/2}|k^{m}+n^{m}|^{(1-2\alpha-\varepsilon)/2}
     |k^{m}+n^{m}|^{\alpha} \\
   & \leq2^{m\alpha}|k^{m}-n^{m}|^{1/2}|k^{m}+n^{m}|^{-\varepsilon/2}
     |j^{m}+n^{m}|^{1/2}|j^{m}-n^{m}|^{-\varepsilon/2}
\end{align*}
and we gets
$$
  R_{m}(n,k,j)\leq
  4^{m}|k^{m}-n^{m}|^{-(1+\varepsilon)/2}|j^{m}-n^{m}|^{-(1+\varepsilon)/2}
$$
and thus obtain estimate of the type \eqref{eq_4}.

In the case $\frac{1}{2}\leq\alpha<1$, since $n\leq|k+n|$,
$$
  <n>^{m(1-\alpha-\varepsilon)}<k>^{m\alpha}\leq 2^{m\alpha}|k^{m}+n^{m}|^{(1-\alpha-
  \varepsilon)}|k|^{m\alpha}
$$
so using that
$<k>^{m(\alpha-1/2)}\leq|k^{m}+n^{m}|^{(\alpha-1/2)}$ and
$<k>^{m/2)}\leq3^{m/2}|k^{m}-n^{m}|$, we get
\begin{align*}
   <n>^{m(1-\alpha-\varepsilon)}<k>^{m\alpha}&\leq 2^{m\alpha}|k^{m}+n^{m}|^{(1-\alpha-
    \varepsilon)}|k|^{m\alpha} \\
  & \leq2^{m\alpha}|k^{m}+n^{m}|^{(1-\alpha-\varepsilon)}3^{m/2}|k^{m}-n^{m}|^{1/2}
    |k^{m}+n^{m}|^{(\alpha-1/2)} \\
  & \leq(2\sqrt{3})^{m}|k^{2m}-n^{2m}|^{1/2}|k^{m}+n^{m}|^{-\varepsilon} \\
  & \leq(2\sqrt{3})^{m}|k^{2m}-n^{2m}|^{1/2}|j^{m}-n^{m}|^{-\varepsilon/2}
    |j^{m}+n^{m}|^{-\varepsilon/2},
\end{align*}
where we use that $|j^{m}\pm n^{m}|\leq|k^{m}+n^{m}|$. This yields
$$
  R_{m}(n,k,j)\leq
  8^{m}|j^{m}+n^{m}|^{-(1+\varepsilon)/2}|j^{m}-n^{m}|^{-(1+\varepsilon)/2}
$$
and therefore we again obtain an estimate of the type
\eqref{eq_4}.
\end{proof}

For later reference, let us denote for given $m\in\mathbb{N}$,
$0\leq\alpha\leq1$ and $R>0$, by $n_{*}=n_{*}(\alpha,R)\geq1$ a
number with the property that, for any $v\in h_{0}^{-m\alpha}$
with $\|v\|_{h^{-m\alpha}}\leq\mathrm{R}$,
\begin{equation}\label{eq_5}
  \sup_{\lambda\in \mathrm{Vert^{m}_{n}(r_{n})}}\|S_{m\lambda}\|
  _{\mathcal{L}(h^{0})}\leq\frac{1}{2}\qquad \forall n\geq n_{*}.
\end{equation}

\section{asymptotics of periodic eigenvalues}\label{sec_2}
In this section we establish asymptotic estimates for the
eigenvalues $\lambda_{n}$. They are obtained by
separately considering the mean $\tau_{mn}$ and the difference $\gamma_{mn}$ of
 a $\lambda_{2n}$ and $\lambda_{2n-1}$,
$$
  \tau_{mn}:=\frac{\lambda_{2n}+\lambda_{2n-1}}{2};\qquad
  \gamma_{mn}:=\lambda_{2n}-\lambda_{2n-1}.
$$

\subsection{Asymptotic Estimates of $\tau_{mn}$} In this
subsection we establish
\begin{proposition}\label{pr_3}
  Let $m\in\mathbb{N}$, $0\leq\alpha<1$, and $\varepsilon>0$.
  Then, uniformly for bounded sets of distributions $v$ in
  $h_{0}^{-m\alpha}$,
  $$
    \tau_{mn}=n^{2m}\pi^{2m}+h^{m(1-2\alpha-\varepsilon)}(n).
  $$
\end{proposition}

The statement of Proposition \ref{pr_3} is a consequence of Lemma
\ref{l_8} and Lemma \ref{l_9} below. Let $R>0$ and
$v\in h_{0}^{-m\alpha}$ with $\|v\|_{h^{-m\alpha}}\leq R$. For $n\geq
n_{*}=n_{*}(\alpha,R)$ with $n_{*}$ chosen as in \eqref{eq_5},
define the Riesz projectors
\begin{align*}
   & P_{n}:=\frac{1}{2\pi i}\int_{\Gamma_{n}}(\lambda-D_{m}-B)^{-1}\,d\lambda, \\
   & P_{n}^{0}:=\frac{1}{2\pi
   i}\int_{\Gamma_{n}}(\lambda-D_{m})^{-1}\,d\lambda,
\end{align*}
where $\Gamma_{n}$ is the positively oriented contour given by $\Gamma_{n}=\{\lambda\in
\mathbb{C}\,|\,|\lambda-n^{2m}\pi^{2m}|=n^{m}\}$. The
corresponding Riesz spaces are the ranges of these projectors
\cite{R1}
$$
  E_{mn}:=P_{n}(h^{-m\alpha});\qquad
  E_{mn}^{0}:=P_{n}^{0}(h^{-m\alpha}).
$$
Both $E_{mn}$ and $E_{mn}^{0}$ are two dimensional subspaces of
$h^{0}$, and $P_{n}$ as well as $(D_{m}+B)P_{n}$ can be considered
as operators from $\mathcal{L}(h^{0})$. Their traces can be computed
to be
$$
  Tr(P_{n})=2;\qquad Tr((D_{m}+B)P_{n})=2\tau_{mn}.
$$
Similarly, we have
$$
  Tr(P_{n}^{0})=2;\qquad Tr(D_{m}P_{n}^{0})=2n^{2m}\pi^{2m}
$$
and thus obtain
$$
  2\tau_{mn}-2n^{2m}\pi^{2m}=Tr((D_{m}+B)P_{n})-Tr(D_{m}P_{n}^{0})=Tr(Q_{mn}),
$$
where $Q_{mn}$ is the operator
$$
  Q_{mn}:=(D_{m}+B-n^{2m}\pi^{2m})P_{n}-(D_{m}-n^{2m}\pi^{2m})P_{n}^{0}.
$$
Substituting the formula for $P_{n}$ and $P_{n}^{0}$ one gets
\begin{align*}
  Q_{mn}  & =\frac{1}{2\pi i}\int_{\Gamma_{n}}(\lambda-n^{2m}\pi^{2m})
            \left((\lambda-D_{m}-B)^{-1}-(\lambda-D_{m})^{-1}\right)\,d\lambda \\
          & =\frac{1}{2\pi i}\int_{\Gamma_{n}}(\lambda-n^{2m}\pi^{2m})
            (\lambda-D_{m}-B)^{-1}B(\lambda-D_{m})^{-1}\,d\lambda.
\end{align*}
Write $Q_{mn}=Q_{mn}^{0}+Q_{mn}^{1}$ with
$$
  Q_{mn}^{0}:=\frac{1}{2\pi i}\int_{\Gamma_{n}}(\lambda-n^{2m}\pi^{2m})
             (\lambda-D_{m})^{-1}B(\lambda-D_{m})^{-1}\,d\lambda,
$$
which leads to the following expression for $\tau_{mn}$,
$$
  \tau_{mn}=n^{2m}\pi^{2m}+\frac{1}{2}Tr(Q_{mn}^{0})+\frac{1}{2}Tr(Q_{mn}^{1}).
$$
\begin{lemma}\label{l_8}
  Let $m\in\mathbb{N}$. For any $v\in h_{0}^{-m\alpha}$ (with
  $0\leq\alpha<1$), $n\geq1$, and $k,l\in\mathbb{Z}$,
\begin{align*}
  Q_{mn}^{0}(k,l)= &
  \begin{cases}
    v(\pm2n)  & \text{if}\quad (k,l)=\pm(n,-n); \\
    0         & \text{otherwise}.
  \end{cases}
\end{align*}
\end{lemma}
\begin{proof}
  For $k,l\in\mathbb{Z}$ and $n\geq1$, we have
\begin{align*}
  Q_{mn}^{0}(k,l) & =\frac{1}{2\pi i}\int_{\Gamma_{n}}(\lambda-n^{2m}\pi^{2m})
             (\lambda-k^{2m}\pi^{2m})^{-1}v(k-l)(\lambda-l^{2m}\pi^{2m})^{-1}\,d\lambda \\
                   & =v(k-l)\frac{1}{2\pi
                   i}\int_{\Gamma_{n }}\frac{\lambda-n^{2m}\pi^{2m}}{(\lambda-k^{2m}\pi^{2m})
                   (\lambda-l^{2m}\pi^{2m})}\,d\lambda
\end{align*}
and the claimed statement follows from
\begin{equation*}
  \frac{1}{2\pi i}\int_{\Gamma_{n}}\frac{\lambda-n^{2m}\pi^{2m}}{(\lambda-k^{2m}\pi^{2m})
  (\lambda-l^{2m}\pi^{2m})}\,d\lambda=
  \begin{cases}
    1 & \text{if}\quad (k,l)\in\{\pm n\}; \\
    0 & \text{otherwise}.
  \end{cases}
\end{equation*}
\end{proof}

Lemma \ref{l_8} implies that
$$
  Tr(Q_{mn}^{0})=0.
$$

Moreover, \emph{range}$(Q_{mn}^{0})\subseteq E_{mn}^{0}$, so since
\emph{range} $(Q_{mn})\subseteq\emph{span}(E_{mn}\cup E_{mn}^{0})$ we conclude
that \emph{range}$(Q_{mn}^{1})\subseteq\emph{span}(E_{mn}\cup
E_{mn}^{0})$ as well. Hence \emph{range}$(Q_{mn}^{1})$ is at most
dimension four and
$$|Tr(Q_{mn}^{1})|\leq 4\|Q_{mn}^{1}\|_{\mathcal{L}(h^{0})}.$$
\begin{lemma}\label{l_9}
  Let $m\in\mathbb{N}$,  $0\leq\alpha<1$, $R>0$ and $\varepsilon>0$.
  Then there exists $C=C(\alpha,\varepsilon)$ so that, for any $v\in
  h_{0}^{-m\alpha}$ with $\|v\|_{h^{-m\alpha}}\leq R$,
  $$
    \left\|\left(\|Q_{mn}^{1}\|_{\mathcal{L}(h^{0})}\right)_{n\geq n_{*}}\right\|
    _{h^{m(1-2\alpha-\varepsilon)}}\leq C\|v\|_{h^{-m\alpha}}^{2},
  $$
where $n_{*}=n_{*}(\alpha,R)$ is given by \eqref{eq_5}.
\end{lemma}
\begin{proof}
  By \eqref{eq_2}, for any $\lambda\in\Gamma_{n}$,
  $(\lambda-D_{m}-B)^{-1}$ is given by
  $D^{-1/2}_{m\lambda}(I_{m\lambda}-S_{m\lambda})^{-1}D^{-1/2}_{m\lambda}$.
  Hence
\begin{align*}
  Q_{mn}^{1} & =\frac{1}{2\pi i}\int_{\Gamma_{n}}(\lambda-n^{2m}\pi^{2m})
            \left((\lambda-D_{m}-B)^{-1}-(\lambda-D_{m})^{-1}B(\lambda-D_{m})^{-1}
            \right)\,d\lambda \\
              & =\frac{1}{2\pi i}\int_{\Gamma_{n}}(\lambda-n^{2m}\pi^{2m})
            D^{-1/2}_{m\lambda}(I_{m\lambda}-S_{m\lambda})^{-1}S_{m\lambda}
            I_{m\lambda}^{-1}S_{m\lambda}I_{m\lambda}^{-1}D^{-1/2}_{m\lambda}\,d\lambda.
\end{align*}
Using \eqref{eq_3} one shows that, for $\lambda\in
Vert^{m}_{n}(r)$ $(n\geq\frac{(8m-4)m}{8m-7}$, $0<r<n^{m}\pi^{2m})$,
\begin{equation}\label{eq_6}
  \|D^{-1/2}_{m\lambda}\|_{\mathcal{L}(h^{0})}\leq r^{-1/2}+\frac{\sqrt{3}}{\pi^{m}}
  n^{-m+1/2}.
\end{equation}
Together with Lemma \ref{l_5} the claimed statement then follows.
\end{proof}
\subsection{Asymptotic Estimates of $\gamma_{mn}$}
To state the asymptotic of $\gamma_{mn}$, introduce, for $v\in
h_{0}^{-m\alpha}$, the sequence $w:=\frac{1}{\pi^{2m}}\frac{v}{k^{m}}*
\frac{v}{k^{m}}$. Notice that for any $n\in\mathbb{Z}$,
$$
  w(2n)=\frac{1}{\pi^{2m}}\sum_{k\neq\pm n}\frac{v(n-k)}{(n-k)^{m}}\cdot
  \frac{v(n+k)}{(n+k)^{m}}.
$$
As $\frac{v}{k^{m}}\in h_{0}^{m(-\alpha+1)}$, the convolution
lemma says that $w\in h^{t}$ with $t=(1-\alpha)\;
(0\leq\alpha<\frac{1}{2})$ and $t<2m(1-\alpha)-1/2$, chosen so
that always $t>-m\alpha$. Also consider the sequence
$(l(2n))_{n\geq1}$,
$$
  l(2n):=\frac{1}{\pi^{2m}}\sum_{k\neq\pm n}\frac{v(n-k)}{n^{m}-k^{m}}\cdot
  \frac{v(n+k)}{n^{m}+k^{m}},
$$
and notice that
$$
  \|l\|_{h^{t}}\leq Const\,\|w\|_{h^{t}}, \qquad t\in\mathbb{R}.
$$
\begin{proposition}\label{pr_4}
  Let $m\in\mathbb{N}$, $0\leq\alpha<1$ and $\varepsilon>0$. Then,
  uniformly on bounded sets of distributions in $h_{0}^{-m\alpha}$,
  $$
    \left(\min_{\pm}\left|\gamma_{mn}\pm2\sqrt{(v+l)(-2n)(v+l)(2n)}\right|\right)_{n\geq 1}\in
    h^{m(1-2\alpha+\varepsilon)}.
  $$
\end{proposition}
\emph{Remark.} An asymptotic estimate only involving $v$ but not
$l$ is of the form
\begin{equation*}
  \left(\min_{\pm}\left|\gamma_{mn}\pm2\sqrt{v(-2n)v(2n)}\right|\right)_{n\geq 1}\in
  \begin{cases}
    h^{m(1/2-\alpha)} & \text{if}\quad 0\leq\alpha<\frac{1}{2}; \\
    h^{m(1-2\alpha-\varepsilon)} & \text{if}\quad
    \frac{1}{2}\leq\alpha<1.
  \end{cases}
\end{equation*}
\begin{proof}
To prove Proposition \ref{pr_4}, consider for
$n_{*}=n_{*}(\alpha,R)$ and $v\in h_{0}^{-m\alpha}$ with $\|v\|_{h^{-m\alpha}}\leq
R$, the restriction $A_{mn}$ of $D_{m}+B-\tau_{mn}$ to the Riesz
space $E_{mn},\;A_{mn}:E_{mn}\longrightarrow E_{mn}$. The
eigenvalues of $A_{mn}$ are
$\lambda_{n}^{\pm}-\tau_{mn}=\pm\frac{\gamma_{mn}}{2}$, hence
$$
  \det(A_{mn})=-(\frac{\gamma_{mn}}{2})^{2}.
$$
We need the following auxiliary result:
\begin{lemma}\label{l_6}
  Let $m\in\mathbb{N}$,  $0\leq\alpha<1$, $R>0$ and $\varepsilon>0$.
  Then there exists $C>0$ so that, for any $v\in
  h_{0}^{-m\alpha}$ with $\|v\|_{h^{-m\alpha}}\leq R$,
\begin{align*}
  \mathrm{(i)}\qquad  & \|P_{n}\|_{\mathcal{L}(h^{0})}\leq C\qquad \forall n\geq n_{*}; \\
  \mathrm{(ii)}\qquad &
  \left(\|P_{n}-P_{n}^{0}\|_{\mathcal{L}(h^{0})}\right)_{n\geq n_*}\in
  h^{m(1-\alpha-\varepsilon)}.
\end{align*}
\end{lemma}
\begin{proof}
  Recall that, for $n\geq n_{*}$ $P_{n}=\frac{1}{2\pi i}\int_{\Gamma_{n}}
D^{-1/2}_{m\lambda}(I_{m\lambda}-S_{m\lambda})^{-1} D^{-1/2}_{m\lambda}\,d\lambda$
and
$$
  P_{n}-P_{n}^{0}=\frac{1}{2\pi i}\int_{\Gamma_{n}}D^{-1/2}_{m\lambda}
  (I_{m\lambda}-S_{m\lambda})^{-1}S_{m\lambda}I_{m\lambda}^{-1}
  D^{-1/2}_{m\lambda}\,d\lambda.\qquad\mbox{  }
$$
The claimed estimates then follow from \eqref{eq_2} and Lemma
\ref{l_5}.
\end{proof}

Choose, if necessery, $n_{*}$ larger so that
\begin{equation}\label{eq_7}
  \|P_{n}-P_{n}^{0}\|_{\mathcal{L}(h^{0})}\leq\frac{1}{2}\qquad \forall\,n\geq
  n_{*}.
\end{equation}
One verifies easily that $Q_{mn}:=(P_{n}-P_{n}^{0})^{2}$ commutes
with $P_{n}$ and $P_{n}^{0}$. Hence $Q_{mn}$ leaves both Riesz
spaces $E_{mn}$ and $E_{mn}^{0}$ invariant. The operator $Q_{mn}$
is used to define, for $n\geq n_{*}$, the restriction of the
transformation operator $(Id-Q_{mn})^{-1/2} (P_{n}P_{n}^{0}+
(Id-P_{n})(Id-P_{n}^{0}))$  to $E_{mn}^{0}$ (cf.\cite{R3}),
$$
  U_{mn}:=(Id-Q_{mn})^{-1/2}P_{n}P_{n}^{0}:E_{mn}^{0}\longrightarrow
  E_{mn},
$$
where $(Id-Q_{mn})^{-1/2}$ is given by the binomial formula
\begin{equation}\label{eq_8}
  (Id-Q_{mn})^{-1/2}=\sum_{l\geq0}
  \left(
  \begin{array}{c}
    -1/2 \\
    l \
  \end{array}
  \right)
  (-Q_{mn})^{l}.
\end{equation}
One verifies that $U_{mn}$ is invertible with the inverse given by
\begin{equation}\label{eq_9}
  U_{mn}^{-1}:=P_{n}^{0}P_{n}(Id-Q_{mn})^{-1/2}.
\end{equation}
As a consequence,
$$
  \det(U_{mn}^{-1}A_{mn}U_{mn})=-(\frac{\gamma_{mn}}{2})^{2}.
$$
To estimate $\det(U_{mn}^{-1}A_{mn}U_{mn})$, write
$$
  U_{mn}^{-1}A_{mn}U_{mn}=P_{n}^{0}P_{n}A_{mn}P_{n}P_{n}^{0}+
  R_{mn}^{(1)}+R_{mn}^{(2)},
$$
where
$$
  R_{mn}^{(1)}:=(U_{mn}^{-1}-P_{n}P_{n}^{0})A_{mn}P_{n}P_{n}^{0}; \qquad
  R_{mn}^{(2)}:=U_{mn}^{-1}A_{mn}(U_{mn}-P_{n}P_{n}^{0}).
$$
The term $P_{n}^{0}P_{n}A_{mn}P_{n}P_{n}^{0}=P_{n}^{0}P_{n} (D_{m}+B-
\tau_{mn})P_{n}^{0}$ is split up further,
\begin{align*}
   P_{n}^{0}P_{n} (D_{m}+B-\tau_{mn})P_{n}^{0}&=P_{n}^{0}(D_{m}+B-
    \tau_{mn})P_{n}^{0}+P_{n}^{0}(P_{n}-P_{n}^{0})(D_{m}+B-\tau_{mn})P_{n}^{0} \\
  & =P_{n}^{0}BP_{n}^{0}+T_{mn}^{(1)}+P_{n}^{0}(P_{n}-P_{n}^{0})B
  P_{n}^{0}+R_{mn}^{(3)},
\end{align*}
where $T_{mn}^{(1)}$ is a diagonal operator (use $D_{mn}P_{n}=n^{2m}\pi^{2m}
P_{n}^{0}$)
$$
  T_{mn}^{(1)}:=P_{n}^{0}(n^{2m}\pi^{2m}-\tau_{mn})P_{n}^{0}
$$
and
$$
  R_{mn}^{(3)}:=P_{n}^{0}(P_{n}-P_{n}^{0})(n^{2m}\pi^{2m}-\tau_{mn})P_{n}^{0}.
$$
As a $2\times2$ matrix, $B_{n}:==P_{n}^{0}B=P_{n}^{0}$ is given by
\begin{equation*}
    \begin{pmatrix}
    B_{n}(n,n)  & B_{n}(n,-n) \\
    B_{n}(-n,n) & B_{n}(-n,-n)
    \end{pmatrix}
=
  \begin{pmatrix}
    0      & v(2n) \\
    v(-2n) & 0
  \end{pmatrix}
.
\end{equation*}
To obtain a satisfactory estemate for
$\det(U_{mn}^{-1}A_{mn}U_{mn})$, we  have to substitute an
expansion of $(P_{n}-P_{n}^{0})$ into $P_{n}^{0}(P_{n}-P_{n}^{0})B
P_{n}^{0}$ and split the main term into a diagonal part
$T_{mn}^{(2)}$ and an off-diagonal part. Let us explain this in
more detail. Write
\begin{align*}
  P_{n}-P_{n}^{0} & =\frac{1}{2\pi i}\int_{\Gamma_{n}}(\lambda-D_{m})^{-1}
                     B(\lambda-D_{m})^{-1}\,d\lambda \\
                   & +\frac{1}{2\pi i}\int_{\Gamma_{n}}(\lambda-D_{m})^{-1}B
                    (\lambda-D_{m})^{-1}B(\lambda-D_{m}-B)^{-1}\,d\lambda,
\end{align*}
which leads to
$$
  P_{n}^{0}(P_{n}-P_{n}^{0})BP_{n}^{0}=\mathcal{L}_{mn}+
  R_{mn}^{(4)},
$$
where
$$
  R_{mn}^{(4)}:=\frac{1}{2\pi i}\int_{\Gamma_{n}}P_{n}^{0}D^{-1/2}_{m\lambda}
  I_{m\lambda}^{-1}S_{m\lambda}I_{m\lambda}^{-1}S_{m\lambda}(I_{m\lambda}-S_{m\lambda})^{-1}
  S_{m\lambda}D^{1/2}_{m\lambda}P_{n}^{0}\,d\lambda
$$
and
\begin{equation}\label{eq_10}
  \mathcal{L}_{mn}:= \frac{1}{2\pi i}\int_{\Gamma_{n}}P_{n}^{0}(\lambda-D_{m})^{-1}B
  (\lambda-D_{m})^{-1}BP_{n}^{0}\,d\lambda.
\end{equation}
As a $2\times2$ matrix,
$\mathcal{L}_{mn}=
  \begin{pmatrix}
    \mathcal{L}_{mn}(n,n)  & \mathcal{L}_{mn}(n,-n) \\
    \mathcal{L}_{mn}(-n,n) & \mathcal{L}_{mn}(-n,-n)
  \end{pmatrix}
$
is of the form
$$
  \mathcal{L}_{mn}=T_{mn}^{(2)}+
  \begin{pmatrix}
    0 & l(2n)\\
    l(-2n) & 0
  \end{pmatrix}
,
$$
where $T_{mn}^{(2)}$ is the diagonal part of $\mathcal{L}_{mn}$.

Combining the computation above, one obtains the following
identity
\begin{equation}\label{eq_11}
  U_{mn}^{-1}A_{mn}U_{mn}=
  \begin{pmatrix}
    0 & (v+l)(2n)\\
    (v+l)(-2n) & 0
  \end{pmatrix}
   +T_{mn}+R_{mn},
\end{equation}
where $T_{mn}$ is the diagonal matrix
$T_{mn}=T_{mn}^{(1)}+T_{mn}^{(2)}$ and $R_{mn}$ is the sum $R_{mn}=
\sum_{j=1}^{4}R_{mn}^{(j)}$. The identity \eqref{eq_11} leads to
the following expression for the determinant
$$
  -\left(\frac{\gamma_{mn}}{2}\right)^{2}=\det(U_{mn}^{-1}A_{mn}U_{mn})=
  -(v+l)(2n)(v+l)(-2n)-r_{mn},
$$
where the error $r_{mn}$ is given by
\begin{align*}
    r_{mn}&=-(T_{mn}(n,n)+R_{mn}(n,n))(T_{mn}(-n,-n)+R_{mn}(-n,-n)) \\
   & +(v+l)(2n)R_{mn}(-n,n)+(v+l)(-2n)R_{mn}(n,-n)+R_{mn}(-n,n)
     R_{mn}(n,-n).
\end{align*}
Hence
$$
  \min_{\pm}\left|\frac{\gamma_{mn}}{2}\pm\sqrt{(v+l)(-2n)(v+l)(2n)}\right|\leq|r_{mn}|^{1/2}.
$$
To estimate $r_{mn}$, use that an entry of a matrix is bounded by
its norm. Hence, for some universal constant $C>0$ and $n\geq
n_{*}$,
\begin{equation}\label{eq_12}
  |r_{mn}|\leq C\left(\|T_{mn}\|^{2}+\|R_{mn}\|^{2}+\sum_{\pm}|
  (v+l)(\pm 2n)|\|R_{mn}\|\right),
\end{equation}
where, for case of notation, $\|\cdot\|$ denotes the operator norm
on $\mathcal{L}(h^{0})$. The terms on the right side of the
inequality above are estimated separately. By Proposition
\ref{pr_3},
$$
  \|T_{mn}^{(1)}\|=|n^{2m}\pi^{2m}-\tau_{mn}|=h^{m(1-2\alpha-\varepsilon)}(n).
$$
As $\|T_{mn}^{(2)}\|=\emph{diag}\,(\mathcal{L}_{mn}(n,n),
\mathcal{L}_{mn}(-n,-n))$ we have $\|T_{mn}^{(2)}\|\leq
\|\mathcal{L}_{mn}\|$ and, by the definition \eqref{eq_10} of
$\mathcal{L}_{mn}$,
\begin{align*}
  \|\mathcal{L}_{mn}\|= & \|\frac{1}{2\pi i}\int_{\Gamma_{n}}P_{n}^{0}D^{-1/2}_{m\lambda}
  I_{m\lambda}^{-1}S_{m\lambda}I_{m\lambda}^{-1}S_{m\lambda}
  D^{1/2}_{m\lambda}P_{n}^{0}\,d\lambda\|\leq \\
                        &
                        n^{m}\left(\sup_{\lambda\in\Gamma_{n}}\|D^{-1/2}_{m\lambda}\|\,
                          \|S_{m\lambda}\|^{2}\,\|D^{1/2}_{m\lambda}P_{n}^{0}\|\right).
\end{align*}
By Lemma \ref{l_5} and \eqref{eq_6} we then conclude
$$
  \|T_{mn}^{(2)}\|=h^{m(1-2\alpha-\varepsilon)}(n).
$$
By the definition of $R_{mn}^{(1)}$,
$$
  \|R_{mn}^{(1)}\|\leq\|U_{mn}^{-1}-P_{n}P_{n}^{0}\|\,\|(D_{mn}+B-\tau_{mn})P_{n}
  \|\,\|P_{n}^{0}\|.
$$
We have $\|P_{n}^{0}\|=1$ and
$$
  \|(D_{mn}+B-\tau_{mn})P_{n}\|=\|\frac{1}{2\pi
  i}\int_{\Gamma_{n}}(\lambda-\tau_{mn})(\lambda-D_{m}-B)^{-1}\,d\lambda\|
  \leq Cn^{m\alpha},
$$
where for the last inequality we use Proposition \ref{pr_2} to
deform the contour $\Gamma_{n}$ to a circle $\Gamma_{n}^{'}$ of radius
$Cn^{m\alpha}$ around $n^{2m}\pi^{2m}$ and the estimate
$\|(\lambda-D_{m}-B)^{-1}\|=\|D^{1/2}_{m\lambda}
(I_{m\lambda}-S_{m\lambda})^{-1}D^{1/2}_{m\lambda}\|\leq
Cn^{-m\alpha}$ $\forall\lambda\in\Gamma_{n}^{'}$. By the formula
\eqref{eq_9} for $U_{mn}^{-1}$, we have, in view of the binomial
formula \eqref{eq_8} and the definition
$Q_{mn}=(P_{n}-P_{n}^{0})^{2}$,
$$
  \|U_{mn}^{-1}-P_{n}P_{n}^{0}\|\leq \|P_{n}\|\,\|P_{n}^{0}\|\,\sum_{l\geq1}
  \begin{pmatrix}
    -1/2 \\
      l
  \end{pmatrix}
  \|P_{n}-P_{n}^{0}\|^{l},
$$
where for the last inequality we use lemma \ref{l_6} (i) and the
estimate $\|P_{n}-P_{n}^{0}\|\leq\frac{1}{2}$ $\;(n\geq n_{*})$.
Hence, by Lemma \ref{l_6} $\mathrm{(ii)}$,
$$
  \|R_{mn}^{(1)}\|\leq Cn^{-m\alpha}\|P_{n}-P_{n}^{0}\|^{2}=
  n^{m\alpha}(h^{m(1-\alpha-\varepsilon)}(n))^{2}.
$$
Similarly one shows
$$
  \|R_{mn}^{(2)}\|=n^{m\alpha}(h^{m(1-\alpha-\varepsilon)}(n))^{2}.
$$
In view of the definition $R_{mn}^{(3)}$,
$$
  \|R_{mn}^{(3)}\|\leq C\|P_{n}-P_{n}^{0}\|\,|n^{2m}\pi^{2m}-\tau_{mn}|=
  n^{m\alpha}(h^{m(1-\alpha-\varepsilon)}(n))^{2},
$$
where we use Lemma \ref{l_6} to estimate $\|P_{n}-P_{n}^{0}\|$ and
Proposition \ref{pr_3} to bound $|n^{2m}\pi^{2m}-\tau_{mn}|$.
Finally, by by the definition of $R_{mn}^{(4)}$ and Lemma
\ref{l_4},
$$
  \|R_{mn}^{(4)}\|\leq Cn^{m}\|S_{mn}\|^{3}\leq
  n^{m}(h^{m(1-\alpha-\varepsilon/2)}(n))^{3}\leq
  n^{m\alpha}(h^{m(1-\alpha-\varepsilon)}(n))^{2}.
$$
Combining the obtained estimates one gets
\begin{equation}\label{eq_13}
  \|R_{mn}\|=n^{m\alpha}(h^{m(1-\alpha-\varepsilon)}(n))^{2},
\end{equation}
\begin{equation}\label{eq_14}
  \|T_{mn}\|=(h^{m(1-2\alpha-\varepsilon)}(n)).
\end{equation}
As a consequence (i.e. $\|l\|_{h^{t}}\leq Const\,\|w\|_{h^{t}}$,
$t>-m\alpha$)
\begin{equation}\label{eq_15}
  |(v+l)(\pm 2n)|\,\|R_{mn}\|=(h^{m(1-2\alpha-\varepsilon)}(n))^{2}
\end{equation}
and, in view of \eqref{eq_10},
$$
  |r_{mn}|^{1/2}=h^{m(1-2\alpha-\varepsilon)}(n).
$$
This proves Proposition \ref{pr_4}.
\end{proof}

\subsection{Spectral Theory for 1-Periodic Distributions}
In this subsection we consider distributions $V\in
H_{per}^{-m\alpha}[-1,1]$ which are of period 1, i.e.
$\hat{V}(2k+1)=0$ for any $k\in\mathbb{Z}$. Let $L_{m}:=(-1)^{m} d^{2m}/dx^{2m}+V$.
As in the case of more regular potentials it holds
\begin{proposition}\label{pr_5}
  Let the $\lambda_{0},\,\lambda_{1},\,\lambda_{2},\,\ldots$ are
  eigenvalues of the differential operator $L_{m}$ with the 1-periodic distribution $V\in
  H_{per}^{-m\alpha}[-1,1]$. Then there exists $n_{0}$ so that
\begin{enumerate}
  \item The eigenvalues $\lambda_{2n-1},\,\lambda_{2n}$ for $n\geq n_{0}$ even
        are eigenvalues of the differential operator $L_{m}$ when considered on
        the interval $[0,1]$ with periodic boundary conditions,
        i.e. with domain $\{f\in H_{per}^{m(2-\alpha)}\,|\,f(x+1)=f(x)\;\forall x\}$.
  \item The eigenvalues $\lambda_{2n-1},\,\lambda_{2n}$ for $n\geq n_{0}$ odd
        are eigenvalues of the  differential operator $L_{m}$ when considered on
        the interval $[0,1]$ with antiperiodic boundary conditions,
        i.e. with domain $\{f\in H_{per}^{m(2-\alpha)}\,|\,f(x+1)=-f(x)\;\forall x\}$.
\end{enumerate}
\end{proposition}
\begin{proof}
This assertions follow from the fact that, for $V$ of period 1,
the operator $D_{m}+B$ (with $B=v*\cdot\;$ and $v(k):=\hat{V}(k)$)
leaves the following two subspaces invariant
\begin{align*}
   & h_{+}^{-m\alpha}:=\{a=(a(k))_{k\in \mathbb{Z}}\,|\,a(2k+1)=0\; \forall
     k\in\mathbb{Z}\},  \\
   & h_{-}^{-m\alpha}:=\{a=(a(k))_{k\in \mathbb{Z}}\,|\,a(2k)=0\; \forall
     k\in\mathbb{Z}\}.
\end{align*}

Hence, \emph{spec}$(D_{m}+B)=$\emph{spec}$^{+}(D_{m}+B)\;\cup$
\emph{spec}$^{-}(D_{m}+B)$ where \emph{spec}$^{\pm}(D_{m}+B)$
denotes the spectrum of $D_{m}+B$ when restricted to
$h_{\pm}^{-m\alpha}$. The claimed statement follows arguing as in
proof Proposition \ref{pr_4}, using that
\begin{align*}
   & \emph{spec}^{+}(D_{m})=\{(2n)^{2m}\pi^{2m}\,|\,n\geq0\}, \\
   & \emph{spec}^{-}(D_{m})=\{(2n+1)^{2m}\pi^{2m}\,|\,n\geq0\}.
\end{align*}
\end{proof}

Proposition \ref{pr_1} can be used to characterize the regularity
of real valued distribution of period 1 by the decay properties of
$(\gamma_{mn})_{n\geq1}$, as claimed in Corollary \ref{cor_1}.
First let us make the following small observation: Let $V$ be a
one-periodic distribution in $H_{per}^{-m\beta}$ for some
$0\leq\beta<1$. Then $V\in H_{per}^{-m\alpha}$ for any
$\beta\leq\alpha\leq1$, and the operator
$(-1)^{m}d^{2m}/dx^{2m}+V$ can be viewed as an unbounded operator
$L_{\alpha}$ on $H_{per}^{-m\alpha}$ with domain
$H_{per}^{m(2-\alpha)}$.
\begin{lemma}\label{l_7}
  The periodic spectrum of $L_{\alpha}$ coincides with the one of $L_{\beta}$ for
  any $\beta\leq\alpha\leq1$.
\end{lemma}
\begin{proof}
Clearly, this is true for $V=0$. To see
that it holds for any arbitrary distribution, note that \emph{spec}$(L_{\beta})
\subseteq$\emph{spec}$(L_{\alpha})$ as the domain of $L_{\alpha}$ is
$H_{per}^{m(2-\alpha)}$ and, for any $\beta\leq\alpha\leq1$, $H_{per}^{m(2-\beta)}\subseteq
H_{per}^{m(2-\alpha)}$. Conversely, it is to be proven that any
eigenfunction $g\in H_{per}^{m(2-\alpha)}$ of $L_{\alpha}$ is in
fact an element in $H_{per}^{m(2-\beta)}$. For an eigenfunction
$g$ of $L_{\alpha}$ one has $Vg\in H_{per}^{-m\beta}$ by Lemma
\ref{col}, hence
$$
  (-1)^{m}\frac{d^{2m}}{dx^{2m}}g=-Vg+\lambda g\in
  H_{per}^{-m\beta}.
$$
Clearly, $\left((-1)^{m}\frac{d^{2m}}{dx^{2m}}+1\right)^{-1}\in
\mathcal{L}\left(H_{per}^{-m\beta},H_{per}^{m(2-\beta)}\right)$
and thus
$$
  g=\left((-1)^{m}\frac{d^{2m}}{dx^{2m}}+1\right)^{-1}
  \left((-1)^{m}g^{(2m)}+g\right)^{-1}\in H_{per}^{m(2-\beta)},
$$
as claimed.
\end{proof}

For convenience let us state Corollary \ref{cor_1} once more.
\begin{corollary}\label{cor_11}
Let $m\in\mathbb{N}$, $0\leq\beta<\alpha<1$. Assume that the
distribution $V\in H^{-m\alpha}_{per}$ is of period 1 and real
valued (i.e. $\hat{V}(2k+1)=0$, $\hat{V}(-k)=\overline{\hat{V}(k)}\quad \forall
k\in\mathbb{Z}$). Then $V\in H^{-m\beta}_{per}$ iff $(\gamma_{mn})_{n\geq 1}\in
h^{-m\beta}$.
\end{corollary}
\begin{proof}
  By Proposition \ref{pr_1} it follows from $V\in H^{-m\beta}_{per}$ that
  $(\gamma_{mn})_{n\geq1}\in h^{-m\beta}$. Conversely, assume that $V\in
  H^{-m\alpha}_{per}$ and $(\gamma_{mn})_{n\geq 1}\in
  h^{-m\beta}$. In view of Proposition \ref{pr_1}, there exists
  $\delta_{*}=\delta(\alpha)>0$, decreasing in $0\leq\alpha<1$,
  such that
\begin{equation}\label{eq_16}
  \left(\min_{\pm}\left|\gamma_{mn}\pm2\sqrt{\hat{V}(-2n)\hat{V}(2n)}\right|
  \right)_{n\geq 1}\in h^{m(-\alpha+\delta_{*})}.
\end{equation}

Using that $\hat{V}(-k)=\overline{\hat{V}(k)}$ and $\gamma_{mn}\geq0$, \eqref{eq_16} becomes
\begin{equation}\label{eq_17}
  \left(\gamma_{mn}-2|\hat{V}(2n)|\right)_{n\geq1}\in
   h^{m(-\alpha+\delta_{*})}.
\end{equation}
As $(\gamma_{mn})_{n\geq 1}\in h^{-m\beta}$, $\hat{V}(-k)=\overline{\hat{V}(k)}\quad
\forall k\in\mathbb{Z}$, and $\hat{V}(2n+1)=0\quad\forall n\in\mathbb{Z}$, one
concludes from \eqref{eq_17} that
$$
  \left(\hat{V}(k)\right)_{k\in\mathbb{Z}}\in
  h^{m(-\alpha+\delta_{1})},
$$
where $\delta_{1}:=\min(\delta_{*},\alpha-\beta)$. If
$\alpha-\beta>\delta_{*}$, we can repeat the above argument.
Taking into acount that $\delta(\alpha)$ is monotone decreasing,
one gets
$$
  \left(\hat{V}(k)\right)_{k\in\mathbb{Z}}\in
  h^{m(-\alpha+\delta_{*}+\delta_{1})},
$$
where $\delta_{2}:=\min(\delta_{*},\alpha-\beta-\delta_{*})$.
After finitely many steps, we get $\left(\hat{V}(k)\right)_{k\in\mathbb{Z}}\in
h^{-m\beta}$, hence $V\in H^{-m\beta}_{per}$.
\end{proof}

\section{appendix: convolution lemma}\label{A}
Given two sequences $a=(a(k))_{k\in \mathbb{Z}}$ and $b=(b(k))_{k\in
\mathbb{Z}}$, the convolution product $a*b$ is formally defined as
the sequence given by
$$
  (a*b)(k):=\sum_{j\in\mathbb{Z}}a(k-j)b(j).
$$
\begin{lemma}[\cite{R2}]\label{col}
  Let $n\in\mathbb{Z}$, $s,r\geq0$, and $\;t\in\mathbb{R}$ with
  $t\leq\min(s,r)$. If $s+r-t>1/2$, than the convolution map is
  continuous (uniformly by in n) when viewed as a map $h^{r,n}\times h^{s,-n}\longrightarrow
  h^{t}$ as well as $h^{-t}\times h^{s,n}\longrightarrow
  h^{-r,n}$.
\end{lemma}


\begin{thebibliography}{99}
\bibitem{R1}
I. Gohberg, M. Krein, Introductia to Theory of Linear
Nonselfadjoint Operators, Nauka Publ., Moscow, 1965 (in Russian).
\bibitem{R2}
T. Kappeler, C. M\"{o}hr, Estimates for Periodic and Dirichlet
Eigenvalues of the Schr\"{o}dinger Operator with Singular
Potentials, J. Func. Anal., $\mathbf{186}$(2001) 62 91.
\bibitem{R3}
T. Kato, "Perturbation Theory for Linear Operators",
Springler-Verlag, Berlin/New York, 1980.
\bibitem{R4}
V. A. Marchenko, Sturm-Liouville Operators and Their Applications,
Naukova Dumka Publ., Kiev, 1977 (in Russian); Engl., transl.:
Birkk\"{a}ser Verlag, Basel, 1986.
\end{thebibliography}
\end{document}